\documentclass[review,onefignum,onetabnum]{article}

\usepackage{latexsym}
\usepackage{amssymb,amsbsy,amsmath,amsfonts,amssymb,amscd,amsthm}
\usepackage{subfigure}
\usepackage{graphicx}
\usepackage{hyperref}
\usepackage{dsfont}
\usepackage{mathtools}
\usepackage{pdfsync}
\usepackage{epsfig,epstopdf}
\usepackage{caption}
\usepackage{xcolor}
\usepackage{hyperref}

\newtheorem{ass}[]{Assumption}
\newtheorem{prpstn}[]{Proposition}
\newtheorem{rmrk}[]{Remark}

\newcommand {\x} { {\bf x} }

\newcommand{\beq}{\begin{equation}}
\newcommand{\eeq}{\end{equation}}

\newcommand\smallO{
  \mathchoice
    {{\scriptstyle\mathcal{O}}}
    {{\scriptstyle\mathcal{O}}}
    {{\scriptscriptstyle\mathcal{O}}}
    {\scalebox{.7}{$\scriptscriptstyle\mathcal{O}$}}
  }

\def\n{{\bf n}}

\newcommand {\e}  {\varepsilon}

\numberwithin{equation}{section}
\numberwithin{prpstn}{section}
\numberwithin{ass}{section}
\numberwithin{rmrk}{section}

\title{Effective interface conditions for continuum mechanical models describing the invasion of multiple cell populations through thin membranes \thanks{This work was partially supported by INdAM, the Italian Ministry of Education, Universities and Research through the MIUR grant ``Dipartimenti di Eccellenza 2018-2022'' (Project no. E11G18000350001) and the Scientific Research Programmes of Relevant National Interest (Project n. 2017KL4EF3).}}

\author{Chiara Giverso\thanks{Politecnico di Torino, IT (chiara.giverso@polito.it)}
\and
Tommaso Lorenzi\thanks{Politecnico di Torino, IT (tommaso.lorenzi@polito.it)}
\and
Luigi Preziosi\thanks{Politecnico di Torino, IT (luigi.preziosi@polito.it)}
}

\begin{document}
\date{}
\maketitle

\begin{abstract}
We consider a continuum mechanical model for the migration of multiple cell populations through parts of tissue separated by thin membranes. In this model, cells belonging to different populations may be characterised by different proliferative abilities and mobility, which may vary from part to part of the tissue, as well as by different invasion potentials within the membranes. The original transmission problem, consisting of a set of mass balance equations for the volume fraction of cells of every population complemented with continuity of stresses and mass flux across the surfaces of the membranes, is then reduced to a limiting transmission problem whereby each thin membrane is replaced by an effective interface. In order to close the limiting problem, a set of biophysically-consistent transmission conditions is derived through a formal asymptotic method. Models based on such a limiting transmission problem may find fruitful application in a variety of research areas in the biological and medical sciences, including developmental biology, immunology and cancer growth and invasion.
\end{abstract}



\section{Introduction}
Transmission problems for nonlinear partial differential equations describing reaction-diffusion processes and transport phenomena in spatial domains that comprise different parts separated by thin layers (i.e. films or membranes) arise in the mathematical modelling of various chemical, physical and biological systems~\cite{abdelkader2013asymptotic, achdou1998effective, aho2016diffusion, ammari2004reconstruction, bellieud2016asymptotic, Scialo, bonnet2016effective, bruna2015effective, ciavolella2020existence, gahn2018effective, geymonat1999mathematical, haddar2008generalized, joly2006matching, lenzi2016anomalous, marigo2016two, neuss2007effective, perrussel2013asymptotic, poignard2012boundary}. \\
Due to the analytical and numerical challenges posed by the presence of such layers~\cite{auvray2018asymptotic}, it is often convenient to approximate the original problem by an equivalent transmission problem whereby each thin layer is replaced by an effective interface. The equivalent problem is then closed by imposing suitable transmission conditions on the effective interfaces. Such an effective-interface approach brings considerable modelling and computational advantages. From the modelling point of view, the main advantage lies in the fact that a detailed model of the phenomena that occur within the thin layer is not required. From the numerical view point, this approximation ensures a strong reduction of computational costs, especially in the case of very thin layers, since it does not require the generation of the fine mesh that would be necessary in order to obtain accurate numerical results inside and in the proximity of the thin layer. The price to pay for having a simpler and computationally more efficient model is the introduction of some effective interface parameters, the estimation of which may require {\it ad hoc} experiments and extensive parameter fitting.

In~\cite{chaplain2019derivation}, we developed an asymptotic method that enables the formal derivation of biophysically-consistent transmission conditions to close such an equivalent transmission problem for a continuum mechanical model of cell invasion through tissues separated by thin porous membranes. The transmission conditions formally derived in~\cite{chaplain2019derivation} can be regarded as a nonlinear generalisation of the classical Kedem--Katchalsky interface conditions~\cite{kedem1958thermodynamic} and, in contrast to other interface conditions of a similar type which have been employed to model cell invasion~\cite{Gallinato}, they allow the cell volume fraction to be discontinuous across the effective interface, while ensuring mass conservation. A rigorous derivation of such transmission conditions, under a few simplifying assumptions and focussing on a specific form of the barotropic relation for the pressure (i.e. the pressure is defined as a power of the cell volume fraction), was recently carried out in~\cite{ciavolella2021effective}.

The methods employed in~\cite{chaplain2019derivation, ciavolella2021effective} apply to biological scenarios in which there is only one single population of invading cells. Yet, physiological and pathological processes of cell invasion through tissues and membranes often involve multiple cell populations, which have different phenotypic characteristics and thus display different behaviours~\cite{Hagedorn_2011}. This might, for instance, be the case of the heterogeneous proliferation and invasion potentials expressed by tumour cells depending on their metabolic characteristics~\cite{astanin,fiandaca2021mathematical,lorenzi2018role,villa2021modeling} or the variability in invasion potential observed in different cell types depending on the deformability of their nucleus~\cite{arduino2015multiphase, giverso2018nucleus}. Hence, modelling such processes requires to describe the spatio-temporal evolution of multiple cell populations, characterised by different proliferation, migration and invasion abilities. For this reason, here we generalise the formal derivation method presented in~\cite{chaplain2019derivation} to the case of multiple cell populations, thus widening the application domain of the transmission conditions resulting therefrom. 

\section{Statement of the problem and main results}
With the aim of extending the results in~\cite{chaplain2019derivation} to encompass a wider range of biological scenarios, we consider a system that comprises $N$ populations of cells which move through a region of space that is filled with a porous embedding medium (e.g. the extracellular matrix). Mathematically, we identify such a region with a simply-connected spatial domain ${\cal D} \subset \mathbb{R}^{d}$, with smooth boundary $\partial \mathcal{D}$, where $d=1,2,3$ depending on the biological problem at hand. We focus on a scenario where the spatial domain is divided into two regions, ${\cal D}_1$ and ${\cal D}_3$, separated by a thin layer, ${\cal D}_2$. The interfaces between the different sub-domains are denoted by $\Sigma_{12}$ and $\Sigma_{23}$, respectively. 

The volume fraction of cells in population $n\in\{1, \ldots, N\}=:{\cal N}$ at position $\x \in {\cal D}_i$ and time $t\geq0$ is modelled by the function $\phi^n_i(t,\x) \geq 0$, whose evolution is governed by the following transmission problem
\begin{equation}\label{tlp}
\left\{
\begin{array}{ll}
\displaystyle{\frac{\partial \phi^n_i}{\partial t} + \nabla \cdot (\phi^n_i \, {\bf v}^n_i)  = \Gamma^n_i(\phi^n_i, \Phi_i)} 
&{\rm in}\ {\cal D}_i, \; i=1,2,3,\\[10pt]
{\bf v}^n_i=-\mu^n_i \nabla [P(\Phi_i)+p^n(\phi^n_i)] 
&{\rm in}\ {\cal D}_i, \; i=1,2,3,\\[10pt]
 [\![ \phi^n {\bf v}^n_i ]\!] \cdot {\bf \n}_{ij} = 0
&{\rm on}\ \Sigma_{ij}, \; i=1,2, \; j=i+1,\\[10pt]
 [\![P(\Phi_i)+p^n(\phi^n_i)]\!] = 0   &{\rm on}\ \Sigma_{ij}, \; i=1,2, \; j=i+1,
\end{array}
\right.
\;\; n\in {\cal N},
\end{equation}
where $\Phi_i := \displaystyle{\sum_{n = 1}^N \phi^n_i}$ is the total cell volume fraction. Equation~\eqref{tlp}$_2$ expresses Darcy's law for cells in a porous environment~\cite{Ambrosi_closure} and
models the cells' tendency to move towards regions where they feel less compressed. The functions $\mu^n_i(t,\x)>0$ are the cell mobility coefficients and the pressure is given by the sum of two terms. The first one, $P$, is a function of the total cell volume fraction, meaning that cells feel the pressure exerted by all nearby cells as a whole. 
Mathematically speaking, this term gives rise to cross-diffusion effects. When $p^n \equiv 0$: if all mobility coefficients were equal, then~\eqref{tlp}$_2$ would reduce to the case of what in mixture theory is called a {\it constrained mixture} (i.e. all populations would move with the same speed); if the mobility coefficients were different, then the ratios between the speeds of different populations would still be fixed and equal to the ratios between the corresponding mobility coefficients. The second term, $p^n$, describes an additional contribution to cell motion and models the fact that cells may still move even if $\nabla P = 0$. The independence of $P$ and $p^n$ from the index $i$ reflects the idea that the cells' response to the pressure does not depend on the sub-domain they are in.

Moreover, the functions $\Gamma^n_i(\phi^n_i, \Phi_i)$ are net growth rates, whose dependence on $\Phi_i$ takes into account the fact that the growth of the $n^{th}$ population in the $i^{th}$ sub-domain may also be affected by the presence of cells belonging to other populations, through the total cell volume fraction. The dependence on the index $n$ of the functions that model the cell mobility coefficient and the net growth rate reflects the fact that cells belonging to different populations may have different mobilities and may divide and die at different rates.

Finally, in the transmission conditions~\eqref{tlp}$_3$ and~\eqref{tlp}$_4$ the vector $\n_{ij}$ is the unit normal to $\Sigma_{ij}$ that points towards ${\cal D}_{j}$, and the notation $[\![ (\cdot) ]\!] := (\cdot)_{j} - (\cdot)_{i}$ is used for the jump of the quantity $(\cdot)$ across the interface $\Sigma_{ij}$. Consistently with the biophysical problem at hand, these transmission conditions ensure, respectively, mass-flux-continuity and stress-continuity across each interface $\Sigma_{ij}$.
  
Note that, since we are only interested in modelling the dynamics of the different cell populations, here we do not describe the dynamics of other constituents of the porous environment in which cells are embedded, such as extracellular fluids and the extracellular matrix. Hence, the mixture is neither saturated nor closed.


Building on~\cite{chaplain2019derivation}, we make the assumptions given hereafter.
\begin{ass}\label{ass:mu}
The functions $\mu^n_i(t,{\bf x})>0$ and $\Gamma^n_i(\phi^n_i, \Phi_i)$ are continuously differentiable. 
\end{ass}
\begin{ass}\label{ass.p}
The functions $P(\Phi_i)$ and $p ^n(\phi^n_i)$ are continuously differentiable and monotonically increasing.
\end{ass}

A crucial point in our study is how continuity of stresses given by condition~\eqref{tlp}$_4$ transfers to continuity of volume fractions of the single cell populations $\phi^n$ and then to continuity of the total cell volume fraction $\Phi$. For this reason, we will divide the presentation of the results of the study into the cases given hereafter.
\begin{description}
\item{\bf Case 1:} $P \equiv 0$ and $p^n\not\equiv 0\ \ \forall n\in {\cal N}$. In this case, under Assumption~\ref{ass.p}, condition~\eqref{tlp}$_4$ implies the continuity of $\phi^n$ across the interfaces between the different sub-domains (i.e. $[\![\phi^n]\!]=0$ for all $n\in {\cal N}$).
\item{\bf Case 2:}
$P\not\equiv 0$ and $p^n \equiv 0\ \ \forall n\in {\cal N}$. In this case, under Assumption~\ref{ass.p}, condition~\eqref{tlp}$_4$ implies the continuity of $\Phi$ across the interfaces between the different sub-domains (i.e. $[\![\Phi]\!]=0$).
\item{\bf Case 3:}
$P\not\equiv 0$ and $p^n \not\equiv 0\ \ \forall n\in {\cal N}$. In this case, under Assumption~\ref{ass.p}, if $\phi^n$ is continuous across the interfaces between the different sub-domains for all $n\in {\cal N}$, and thus $\Phi$ is continuous as well, then condition~\eqref{tlp}$_4$ will certainly be satisfied. However, due to the generality of the functions $P$ and $p^n$, the existence of situations in which condition~\eqref{tlp}$_4$ is satisfied but $\phi^n$ is not continuous across the interfaces cannot be excluded a priori. This is why, when considering this case, we require the following assumption to hold:
\begin{ass}\label{ass:cont}
The  system resulting from imposing $[\![P(\Phi)+p^n(\phi^n)]\!]=0$ for $n\in{\cal N}$ is satisfied only if $[\![\phi^n]\!]=0$.
\end{ass}
\end{description}


In most biologically relevant scenarios arising in the study of cell invasion through thin layers, such as porous membranes and tissue monolayers, the thickness of the layer is much smaller than the characteristic size of the spatial domain in which the cells are contained. Therefore, defining the thickness of the layer modelled by ${\cal D}_2$ as $\displaystyle{\e := \max_{\hat{\x}_{12} \in \Sigma_{12}} \big\{ \min_{\n_{12}} \{a > 0: \hat{\x}_{12} + a \, {\bf n}_{12} \in \Sigma_{23}\} \big\}}$, we rewrite the transmission problem~\eqref{tlp} as
\begin{equation}\label{tlpeps}
\left\{
\begin{array}{ll}
\displaystyle{\frac{\partial \phi^n_{i \e}}{\partial t} - \nabla \cdot\{\mu^n_{i \e} \, \phi^n_{i \e} \, \nabla [P(\Phi_{i\e})+p^n(\phi^n_{i\e}) ]\} = \Gamma^n_{i}(\phi^n_{i \e}, \Phi_{i \e})} 
&{\rm in}\ {\cal D}_{i \e}, \; i=1,2,3,\\[10pt]
\mu^n_{i \e} \, \phi^n_{i \e} \, \nabla [P(\Phi_{i\e})+p^n(\phi^n_{i\e}) ] \cdot {\bf n}_{ij} =
\mu^n_{j \e} \, \phi^n_{j \e} \, \nabla [P(\Phi_{j\e})+p^n(\phi^n_{j\e}) ] \cdot {\bf n}_{ij} &{\rm on}\ \Sigma_{ij \e}, \; i=1,2, \; j=i+1,\\[10pt]
P(\Phi_{i\e})+p^n(\phi^n_{i\e}) = P(\Phi_{j\e})+p^n(\phi^n_{j\e}) &{\rm on}\ \Sigma_{ij \e}, \; i=1,2, \; j=i+1,
\end{array}
\right.
\;\; n \in {\cal N},
\end{equation}
where $\Phi_{i \e} := \displaystyle{\sum_{n = 1}^N \phi^n_{i \e}}$. Then, we wish to replace the sub-domain ${\cal D}_{2 \e}$ with an effective interface $\tilde \Sigma_{13}$, which is obtained from the actual interfaces $\Sigma_{12 \e}$ and $\Sigma_{23 \e}$ by letting $\e \to 0$, and consider, instead of the transmission problem~\eqref{tlpeps}, an effective interface problem of the following form
\begin{equation}\label{thinlpo}
\left\{
\begin{array}{ll}
\displaystyle{\frac{\partial \tilde \phi^n_i}{\partial t} - \nabla \cdot\{\tilde \mu^n_i \, \tilde \phi^n_i \, 
\nabla [P(\tilde\Phi_{i})+p^n(\tilde \phi^n_{i}) ]\} = \Gamma^n_i(\tilde \phi^n_{i},\tilde \Phi_{i})} 
&{\rm in}\ \tilde {\cal D}_i, \quad i=1,3,\\[10pt]
\text{transmission conditions} &{\rm on}\ \tilde \Sigma_{13},
\end{array}
\right.
\;\; n\in{\cal N},
\end{equation}
with $\displaystyle{\tilde {\cal D}_i := \lim_{\e \to 0} {\cal D}_{i \e}}$, $\displaystyle{\tilde \Sigma_{13} := \lim_{\e \to 0} \Sigma_{12 \e} = \lim_{\e \to 0} \Sigma_{23 \e}}$ and $\displaystyle{\tilde {\mu}^n_i := \lim_{\e \to 0} {\mu}^n_{i \e}}$, $\displaystyle{\tilde \phi^n_i := \lim_{\e \to 0} \phi^n_{i \e}}$ and $\displaystyle{\tilde \Phi_{i} := \displaystyle{\sum_{n = 1}^N \tilde \phi^n_{i}}}$ for $i=1,3$.
This requires us to find biophysically-consistent transmission conditions to complete the effective interface problem~\eqref{thinlpo}. The transmission conditions derived here (see Proposition~\ref{Th1}and Remark 1 below) apply to the case in which the following natural conditions hold
\begin{equation}
\label{eq:assTC2}
\mu^n_{2 \e}\xrightarrow[\e \to 0]{} 0 \; \text{ in such a way that } \; \frac{\mu^n_{2 \e}}{\e} \xrightarrow[\e \to 0]{} \tilde \mu^n_{13}, \quad n\in{\cal N},
\end{equation}
\begin{equation}
\label{eq:assTC2b}
{\rm and}\qquad \lim_{\e \to 0} \frac{\nabla \mu^n_{2 \e}}{\e} \cdot {\n}_{12}  = \lim_{\e \to 0} \frac{\nabla \mu^n_{2 \e}}{\e} \cdot {\n}_{23}  = 0, \quad n\in{\cal N},
\end{equation}
where $\tilde \mu^n_{13}$ is a real, positive and bounded function that can be seen as the {\it effective mobility coefficient} of cells in the $n^{th}$ population through the thin membrane modelled by the effective interface $\tilde \Sigma_{13}$. In the remainder of this section, we will use the notation $\tilde{\n}_{13}$ for the unit vector normal to the interface $\tilde \Sigma_{13}$ that points towards the sub-domain $\tilde {\cal D}_{3}$. 

\begin{prpstn}\label{Th1} 
Under Assumptions \ref{ass:mu} and \ref{ass.p}, the following set of transmission conditions 
\begin{equation}
\label{eq:TCflux}
\tilde \mu^n_1 \, \tilde \phi^n_1  \nabla [P(\tilde \Phi_1)+p^n(\tilde \phi^n_1)] \cdot \tilde{\n}_{13}= 
\tilde \mu^n_3 \, \tilde \phi^n_3  \nabla [P(\tilde \Phi_3)+p^n(\tilde \phi^n_3)] \cdot\tilde{\n}_{13} 
\quad  \text{ on } \tilde \Sigma_{13}, \quad n\in{\cal N}
\end{equation}
formally applies to the effective interface problem~\eqref{thinlpo}. Moreover, assume that conditions~\eqref{eq:assTC2} and~\eqref{eq:assTC2b} hold as well and define 
\begin{equation}\label{Pi}
\Pi^{'}(\Phi)  := \Phi \, P'(\Phi)\quad{\rm and}\quad \pi^n{'}(\phi^n)  := \phi^n p^n{'}(\phi^n).
\end{equation} 
Then, formally,
\begin{description}
\item{$\bullet$} in Case 1 (i.e. when $P \equiv 0$ and $p^n \not\equiv 0\ \ \forall n\in {\cal N}$), the transmission conditions of the effective interface problem~\eqref{thinlpo} are
\begin{equation}\label{eq:TC1}
\tilde\mu_{13}^{n}[\![\pi^n]\!]
=\tilde\mu^n_1\tilde\phi^n_1\nabla p^n(\tilde\phi^n_1)\cdot\tilde{\n}_{13}
=\tilde\mu^n_3\tilde\phi^n_3\nabla p^n(\tilde\phi^n_3)\cdot\tilde{\n}_{13}
 \quad \text{ on } \; \tilde \Sigma_{13}, \quad n\in{\cal N};
\end{equation}
\item{$\bullet$} in Case 2 (i.e. when $P \not\equiv 0$ and $p^n \equiv 0\ \ \forall n\in {\cal N}$), the transmission conditions of the effective interface problem~\eqref{thinlpo} are given by the set of transmission conditions~\eqref{eq:TCflux} alongside
\begin{equation}\label{eq:TC2}
\left[\!\!\left[\Pi(\Phi)\right]\!\!\right] 
=\sum_{n=1}^N \dfrac{\tilde\mu^n_1}{\tilde\mu^n_{13}}\tilde\phi^n_1\nabla P(\tilde \Phi_1)\cdot\tilde{\n}_{13}
=\sum_{n=1}^N \dfrac{\tilde\mu^n_3}{\tilde\mu^n_{13}}\tilde\phi^n_3\nabla P(\tilde \Phi_3)\cdot\tilde{\n}_{13}
 \quad \text{ on } \; \tilde \Sigma_{13};
\end{equation}
\item{$\bullet$} in Case 3 (i.e. when $P \not\equiv 0$ and $p^n \not\equiv 0\ \ \forall n\in {\cal N}$), under Assumption~\ref{ass:cont}, the transmission conditions of the effective interface problem~\eqref{thinlpo} include the set of transmission conditions~\eqref{eq:TCflux} and must satisfy
\begin{equation}
\label{eq:TC3}
\left[\!\!\left[\Pi(\Phi)+\sum_{n=1}^N\pi^n(\phi^{n})\right]\!\!\right] 
=\sum_{n=1}^N 
\dfrac{\tilde\mu^n_1}{\tilde\mu^n_{13}}\tilde\phi^n_1\nabla [P(\tilde \Phi_1)+p^n(\tilde \phi^n_1)]\cdot\tilde{\n}_{13}
=\sum_{n=1}^N 
\dfrac{\tilde\mu^n_3}{\tilde\mu^n_{13}}\tilde\phi^n_3\nabla [P(\tilde \Phi_3)+p^n(\tilde \phi^n_3)]\cdot\tilde{\n}_{13}
 \quad \text{ on } \; \tilde \Sigma_{13}.
\end{equation}
\end{description}
\end{prpstn}

\begin{proof}
For ease of presentation, we formally derive the required transmission conditions in the case where $\Sigma_{12 \e}$ and $\Sigma_{23 \e}$ are parallel planes. The formal derivation carried out here could be extended to more general cases, provided that the sub-domain ${\cal D}_{2 \e}$ and the interfaces $\Sigma_{12 \e}$ and $\Sigma_{23 \e}$ are sufficiently smooth. In fact, since in these more general cases the curvature of the interfaces would appear in the asymptotic expansions employed in the derivation, we expect this formal procedure to fail when the curvature becomes singular. 

We introduce the notation ${\cal D}_{2 \e} \ni \x := (x_\perp, \mathbf{x}_{\Sigma})$, where $x_\perp := \x \cdot \n_{12} $. We also make the change of variables $x_{\perp} \mapsto x_{\perp} - \hat{x}_{12 \perp}$, with $\hat x_{12 \perp}$ given by $\hat{\x}_{12} = (\hat x_{12 \perp}, \hat \x_{12 \Sigma}) \in \Sigma_{12 \e}$, and let $\eta := \dfrac{x_{\perp}}{\e}  \in (0,1)$. Moreover, building on~\cite{chaplain2019derivation}, we make the ansatz $\phi^n_{2\varepsilon}\left(\dfrac{x_{\perp}}{\e}, \mathbf{x}_{\Sigma}\right) = \phi^{n 0}_{2}(\eta, \mathbf{x}_{\Sigma}) + \varepsilon \, \phi^{n 1}_{2}(\eta, \mathbf{x}_{\Sigma}) + \smallO(\varepsilon)$ so that $\Phi_{2\varepsilon}\left(\dfrac{x_{\perp}}{\e}, \mathbf{x}_{\Sigma}\right) = \Phi^{0}_{2}(\eta, \mathbf{x}_{\Sigma}) + \varepsilon \, \Phi^{1}_{2}(\eta, \mathbf{x}_{\Sigma}) + \smallO(\varepsilon)$. Substituting in~\eqref{tlpeps}$_1$ formally gives, at the leading order,
\beq\label{leading}
\frac{\partial}{\partial \eta} \left\{\tilde\mu^n_{13} \, \phi^{n0}_{2} \dfrac{\partial }{\partial \eta}[P(\Phi_{2}^0)+p^n(\phi^{n0}_{2})]\right\}= 0 \; \Longrightarrow \; \tilde\mu^n_{13} \, \phi^{n0}_{2} \dfrac{\partial }{\partial \eta}[P(\Phi_{2}^0)+p^n(\phi^{n0}_{2})] =C_n \; \forall \eta \in (0,1), \; n\in{\cal N},
\eeq
where $C_n$ are real constants. Moreover, substituting in~\eqref{tlpeps}$_2$ formally gives, at the leading order,
\begin{equation}
\label{res2}
\begin{cases}
\tilde\mu^n_{13} \phi^{n0}_{2} \dfrac{\partial}{\partial \eta}[P(\Phi_{2}^{0})+p^n(\phi^{n0}_{2}) ]\Big|_{\eta=0} = \tilde \mu^n_{1} \, \tilde \phi^{n}_{1}\nabla [P(\tilde \Phi_{1})+p^n(\tilde \phi^n_{1}) ] \cdot \tilde{\n}_{13} \Big|_{\tilde \Sigma_{13}}, \\[10pt]
\tilde\mu^n_{13} \phi^{n0}_{2} \dfrac{\partial}{\partial \eta}[P(\Phi_{2}^0)+p^n(\phi^{n0}_{2}) ]\Big|_{\eta=1} = \tilde \mu^n_{3} \,  \tilde \phi^{n}_{3}  \nabla [P(\tilde \Phi_{3})+p^n(\tilde \phi^{n}_{3}) ] \cdot \tilde{\n}_{13} \Big|_{\tilde \Sigma_{13}},
\end{cases}
\qquad n\in{\cal N}.
\end{equation}
Letting ${\bf Q}_i^n := \tilde\mu_{i}^n \tilde \phi_{i}^{n} \nabla[P(\tilde \Phi_{i})+p^n(\tilde \phi_{i}^{n})]$ be the flux of the $n^{th}$ population in the sub-domain $\tilde {\cal D}_i$ with $i=1,3$, conditions~\eqref{leading} and~\eqref{res2} merge into 
$$
{\bf Q}_1^n\cdot\tilde{\n}_{13}\Big|_{\tilde \Sigma_{13}}={\bf Q}_3^n\cdot\tilde{\n}_{13}\Big|_{\tilde \Sigma_{13}}, \quad n\in{\cal N},
$$
which implies that the transmission conditions~\eqref{eq:TCflux} formally hold, and 
\begin{equation}\label{condition}
{\bf Q}_1^n\cdot\tilde{\n}_{13}\Big|_{\tilde \Sigma_{13}}={\bf Q}_3^n\cdot\tilde{\n}_{13}\Big|_{\tilde \Sigma_{13}}=
\tilde\mu_{13}^{n}\phi_{2}^{n0} \dfrac{\partial}{\partial\eta}[P(\Phi_{2}^0)+p^n(\phi_{2}^{n0})]=C_n\,, \quad n\in{\cal N}.
\end{equation}
Hence, using the fact that, under the assumptions considered here, $\tilde{\mu}^n_{13}$ is a constant function with respect to $\eta$, while its value could vary in $\mathbf{x}_{\Sigma}$, for all $n \in {\cal N}$, we formally have
\beq\label{C_n}
C_n=\int_0^1 C_n\,{\rm d}\eta = \int_0^1\tilde\mu_{13}^{n} \, \phi_{2}^{n0} \dfrac{\partial}{\partial\eta}\left[P(\Phi_{2}^0)+p^n(\phi_{2}^{n0}) \right]\,{\rm d}\eta =\tilde\mu_{13}^{n} \int_0^1\left[\phi_{2}^{n0} \dfrac{\partial}{\partial\eta}P(\Phi_{2}^0)+\dfrac{\partial}{\partial\eta} \pi^n(\phi_{2}^{n0})\right]\,{\rm d}\eta,
\eeq
with $\pi^n$ defined via~\eqref{Pi}. Furthermore, dividing both sides of~\eqref{C_n} by $\tilde \mu^{n}_{13}>0$ and summing over $n$ gives
\beq\label{condition_senza}
\sum_{n=1}^N\frac{C_n}{\tilde\mu_{13}^n}=
\int_0^1 \, \left[\Phi^{0}_{2} \frac{\partial}{\partial \eta}P(\Phi^0_{2})+
\sum_{n=1}^N \dfrac{\partial}{\partial\eta} \pi^n(\phi_{2}^{n0}) \right]\, {\rm d}\eta =\int_0^1 \frac{\partial  }{\partial \eta}\left[\Pi(\Phi^0_{2}) + \sum_{n=1}^N \pi^n(\phi_{2}^{n0})\right]\, {\rm d}\eta, 
\eeq
with $\Pi$ defined via~\eqref{Pi}. We will now consider Cases 1-3 separately. 

In Case 1 (i.e. when $P \equiv 0$ and $p^n \not\equiv 0\ \ \forall n\in {\cal N}$), \eqref{C_n} reduces to
\beq\label{C_n1}
C_n=\tilde\mu_{13}^{n} \int_0^1 \dfrac{\partial}{\partial\eta}\pi^n(\phi_{2}^{n0})\,{\rm d}\eta=\tilde\mu_{13}^{n} \left[\pi^n(\phi_{2}^{n0})\Big|_{\eta=1}-\pi^n(\phi_{2}^{n0})\Big|_{\eta=0}\right], \quad n\in{\cal N}.
\eeq
Since, as previously noted, in Case 1 we have $[\![\phi^n]\!]=0$ for all $n\in {\cal N}$, combining~\eqref{C_n1} with~\eqref{condition} yields \eqref{eq:TC1}. 

In Case 2 (i.e. when $P \not\equiv 0$ and $p^n \equiv 0\ \ \forall n\in {\cal N}$), \eqref{condition_senza} reduces to
\beq\label{C_n2}
\sum_{n=1}^N\frac{C_n}{\tilde\mu_{13}^n}= \int_0^1 \dfrac{\partial}{\partial \eta}\Pi(\Phi^0_{2}) \,{\rm d}\eta=\Pi(\Phi^0_{2})\Big|_{\eta=1}- \Pi(\Phi^0_{2})\Big|_{\eta=0}.
\eeq
Since, as previously noted, in Case 2 we have $[\![\Phi]\!]=0$, combining~\eqref{C_n2} with~\eqref{condition} yields \eqref{eq:TC2}. 


In Case 3 (i.e. when $P \not\equiv 0$ and $p^n \not\equiv 0\ \ \forall n\in {\cal N}$), \eqref{condition_senza} implies that
\beq\label{C_n3}
\sum_{n=1}^N\frac{C_n}{\tilde\mu_{13}^n}=
\Pi(\Phi^0_{2})\Big|_{\eta=1}- \Pi(\Phi^0_{2})\Big|_{\eta=0} + \sum_{n=1}^N \left[\pi^n(\phi_{2}^{n0})\Big|_{\eta=1}- \pi^n(\phi_{2}^{n0})\Big|_{\eta=0}\right].
\eeq
Under Assumption~\ref{ass:cont}, combining~\eqref{C_n3} with~\eqref{condition} yields \eqref{eq:TC3}. 
\end{proof}

Note that a method of proof similar to the one used for Case 3 in Proposition~\ref{Th1} would allow one to show that if $P \not\equiv 0$ and $p^n \equiv 0$ only for some $n\in {\cal N}$, then \eqref{eq:TC3} formally holds even if Assumption~\ref{ass:cont} is not satisfied. In fact, in this case one would need that $[\![\Phi]\!]=0$ and $[\![\phi^n]\!]=0$ only for those $n$ such that $p^n\not\equiv0$, and these continuity conditions would hold under the other assumptions of Proposition~\ref{Th1}.

\begin{rmrk}
Note that conditions~\eqref{eq:TCflux} ensure continuity of mass fluxes across the effective interface $\tilde \Sigma_{13}$. Moreover, the other conditions provided by Proposition~\ref{Th1} allow one to close the effective interface problem~\eqref{thinlpo} in Cases 1 and 2.
In fact, in Case 1 conditions~\eqref{eq:TC1} are simply $2N$ transmission conditions for a system of $N$ parabolic equations. In Case 2, conditions~\eqref{eq:TCflux} and \eqref{eq:TC2} define a set of $N+1$ conditions on the interface $\tilde \Sigma_{13}$ and the system~\eqref{thinlpo} can be rewritten as $2$ parabolic equations, which require $2$ conditions on $\tilde \Sigma_{13}$, coupled with $2(N-1)$ hyperbolic equations, which require $N-1$ conditions on $\tilde \Sigma_{13}$. On the other hand, if $P\not\equiv 0$ and $p^n \not\equiv 0$ for all or some $n \in {\cal N}$, then conditions~\eqref{eq:TCflux} and \eqref{eq:TC2} are not sufficient for closing problem~\eqref{thinlpo}, since more than $N+1$ conditions on $\tilde \Sigma_{13}$ are needed. This is a classical situation in mixture theory whereby one needs to transfer transmission conditions for the overall mixture onto the single constituents. For this to be done, further assumptions are required. In any case, the transmission conditions must satisfy~\eqref{eq:TC2}. This can be ensured, for instance, by imposing~\eqref{eq:TCflux} alongside 
$$
\tilde\mu_{13}^{n}\{\alpha^n\Pi(\tilde\Phi_3)-\beta^n\Pi(\tilde\Phi_1)+[\![\pi^n(\tilde\phi^{n})]\!] \} =\tilde\mu^n_1\tilde\phi^n_1\nabla [P(\tilde\Phi_1)+p^n(\tilde\phi^n_1)]\cdot\tilde{\n}_{13}, \quad n\in{\cal N},
$$
with $\displaystyle{\sum_{n=1}^N \alpha^n=\sum_{n=1}^N \beta^n=1}$, so that dividing by $\tilde\mu_{13}^{n}$ and summing over $n$ yields~\eqref{eq:TC2}.
\end{rmrk}

\section{Concluding remarks}
Starting from a continuum mechanical model for the dynamics of multiple cell populations that migrate through different parts of tissue separated by a thin membrane, which is represented as a finite region of small thickness, we have introduced a limiting transmission problem,  whereby the membrane is replaced by an effective interface, and derived a set of biophysically-consistent interface conditions to close the limiting problem.

The effective mobility coefficient $\tilde \mu^n_{13}$ of the cells of population $n$ through the thin membrane modelled by the effective interface $\tilde \Sigma_{13}$ can be related to the size of the pores of the membrane, and to the geometrical and mechanical characteristics of the cells, as similarly done in~\cite{arduino2015multiphase, giverso2018nucleus, giverso2014influence}. This makes the limiting transmission problem~\eqref{thinlpo} suitable for providing a possible macroscopic description of cell invasion through thin membranes that takes explicitly into account cell microscopic characteristics, such as the mechanical constraints imposed by the cell nuclear envelope and the solid material inside it~\cite{wolf2013physical}. 

Mathematical models based on the effective interface problem~\eqref{thinlpo} may be fruitfully applied to the study of a variety of biological processes in which cells characterised by different proliferative abilities and mobility properties (e.g. cells of different types, cells of the same type but expressing different phenotypic characteristics, cells displaying different behaviours connected with the cell cycle or the interaction with the extracellular environment) invade thin membranes or tissue monolayers. These processes include immune surveillance and pathological conditions,  such as cancer invasion and fibrosis, which will be the object of study of future work.

We conclude by stressing the fact that carrying out a rigorous derivation of the transmission conditions provided by Proposition~\ref{Th1} remains an open problem -- which is particularly challenging in Cases 2 and 3 due the way in which equations~\eqref{tlpeps}$_1$ for $n \in \mathcal{N}$ are coupled through $\nabla P$ -- and a strong dependence of the convergence space on the form of the functions $P$ and $p^n$ is to be expected. Moreover, \emph{ad hoc} numerical schemes need to be identified, depending on the nature of the effective interface problem~\eqref{thinlpo} and on the choice of the functions $P$ and $p^n$, in order to carry out numerical simulations.

\bibliographystyle{siam}
\bibliography{biblio}

\end{document}